\documentclass[draft,11pt,twoside,nolimits,envcountsame,envcountresetchap]{svmult}
\usepackage{amsmath,amsfonts}
\usepackage{enumerate}
\usepackage{mathrsfs}
\usepackage{eucal}

\def\nm{\|\kern-1pt|}
\def\Ce{\mathrm{C}}

\def\AP{\mathrm{UAP}}
\def\Cb{\mathrm{BC}}
\def\BC{\Cb}
\def\BUC{\mathrm{BUC}}
\def\Ell{\mathrm{L}}
\def\Bb{\mathrm{B}}
\def\BVb{\mathrm{BV}}
\def\co{\mathrm{c}_0}
\def\Lip{\mathrm{Lip}}

\def\aap{\mathrm{aap}}

\def\Lipb{\Lip_b}
\def\SSS{\mathrm{S}}
\def\BV{\mathrm{BV}}

\def\BVb{\BV_b}

\def\MM{\mathrm{M}}
\def\MMc{\MM_c}

\def\LLL{\mathscr{L}}
\def\Id{\mathrm{I}}
\def\ran{\mathrm{rg}}

\def\pow{\mathscr{P}}
\def\FC{\mathcal{F}}
\def\GFC{\mathcal{G}}

\def\id{\mathrm{id}}
\def\clran{\overline{\mathrm{rg}}}
\def\clconv{\overline{\mathrm{conv}}}

\def\NN{\mathbb{N}}
\def\QQ{\mathbb{Q}}
\def\ZZ{\mathbb{Z}}
\def\RR{\mathbb{R}}

\def\Ball{\mathrm{B}}

\def\dd{\mathrm{d}}
\def\ee{\mathrm{e}}

\def\Diff{\Delta}

\def\eqegy{=}

\def\rest{\mid}

\numberwithin{theorem}{section}
\newtheorem{condition}[theorem]{Condition}
\numberwithin{subsubsection}{section}
\begin{document}

\smartqed
\title*{The periodic decomposition problem}
\author{B\'alint Farkas and Szil\'ard Gy. R\'ev\'esz}
\institute{Bergische Universit\"at Wuppertal, Germany \email{farkas@math.uni-wuppertal.de}
\and Alfr\'ed R\'enyi Institute of Mathematics, Hungarian Academy of Sciences; Kuwait University, Kuwait \email{szilard.revesz@renyi.hu} \email{szilard@sci.kuniv.edu.kw}}

\motto[0.75\hsize]{Dedicated to Imre Z.{} Ruzsa on the occasion  of his 60$^\text{th}$ birthday.}
\maketitle

\begingroup
\def\thefootnote{\hskip-0.5em}
\footnotetext{B.F. was supported in part by the Hungarian National Foundation for Scientific Research, Project \#K-100461, Sz.R. was supported in part by the Hungarian National Foundation for Scientific Research, Project \#'s
K-81658, 
K-100461, 
NK-104183, 
K-109789. 
}

\endgroup
\abstract{If a function $f:\RR\to\RR$ can be represented as the sum of $n$ periodic functions as $f=f_1+\dots+f_n$ with $f(x+\alpha_j)=f(x)$ ($j=1,\dots,n$), then it also satisfies a corresponding $n$-order difference equation $\Delta_{\alpha_1}\dots\Delta_{\alpha_n} f=0$. The periodic decomposition problem asks for the converse implication, which  may hold or fail depending on the context (on the system  of periods, on the function class in which the problem is considered, etc.). The problem has natural extensions and ramifications in various directions, and is related to several other problems in real analysis, Fourier and functional analysis. We give a survey about the available methods and results, and present a number of intriguing open problems.
}



\section{Introduction}\label{sec:Intro}

Let $f:\RR\to \RR$ be a function with
\begin{equation}\label{eq:periodd}
f=f_1+\cdots +f_n,\qquad f_j(x+\alpha_j)=f_j(x)
\quad{\forall x\in\RR, j=1,\dots,n},
\end{equation}
where $\alpha_j\in\RR$ are fixed real numbers, we call this an \emph{$(\alpha_1,\dots,\alpha_n)$-periodic decomposition} of $f$. For $\alpha\in\RR$ let
$\Diff_\alpha$ denote the (forward) \emph{difference operator}
\begin{equation*}
\Diff_\alpha:\RR^\RR\to\RR^\RR,\quad\Diff_\alpha g(x):=g(x+\alpha)-g(x).
\end{equation*}
Then the $\alpha_i$-periodicity of $f_i$ above means
$\Diff_{\alpha_i}f_i\eqegy0$. The difference operators commute,
so
\begin{equation}\label{eq:kern}
\Diff_{\alpha_1}\Diff_{\alpha_2}\dots\Diff_{\alpha_n}f \eqegy 0.
\end{equation}
\begin{problem}[I.{}Z.{} Ruzsa, 70s]\label{p:Ruzsa} Does the converse implication
``\eqref{eq:kern} $\Rightarrow$ \eqref{eq:periodd}'' hold true?
\end{problem}
Naturally, this question can be posed \emph{in any given function class} $\FC\subseteq \RR^{\RR}$.
\begin{definition}\label{d:dprall}
Let $\FC\subseteq \RR^\RR$ be a set of functions. With $n\in \NN$, $n\geq1$, and $\alpha_1,\dots,\alpha_n\in \RR$ given, the function class $\FC$ is said to have the \emph{decomposition property with respect to $\alpha_1,\dots,\alpha_n$} if for each $f\in \FC$  satisfying \eqref{eq:kern} a periodic decomposition \eqref{eq:periodd} exists with $f_j\in \FC$ ($j=1,\dots,n$).
Furthermore, the function class $\FC$ has the $n$-\emph{decomposition property} if it has the decomposition property for every choice of $\alpha_1\dots, \alpha_n\in \RR$, and $\FC$ has the \emph{decomposition property} if it has the $n$-decomposition property for each integer $n\geq1$.
\end{definition}

Note that $\RR^\RR$ or $\Ce(\RR)$ (space of continuous functions)  do \emph{not} have the $n$-decomposition property for $n\geq 2$.  Indeed, let $n=2$ and $\alpha_1=\alpha_2=\alpha$. The identity function $\id (x):=x$ satisfies $\Delta_{\alpha} \Delta_{\alpha} \id
\eqegy 0$, but it fails to be $\alpha$-periodic. So the implication ``\eqref{eq:kern} $\Rightarrow$ \eqref{eq:periodd}'' fails. As a matter of fact, a function class containing the identity does not have the decomposition property.

The above choice for $\alpha_1,\alpha_2$ hides the nature of the problem a bit: The existence of periodic decomposition may depend on the system $\alpha_1,\dots,\alpha_n$ of prescribed periods. If we take $\alpha_1=1$ and $\alpha_2=\sqrt{2}$ the arguments above do not work. And in fact, \emph{if $\alpha_1$ and $\alpha_2$ are incommensurable} (i.e., $\alpha_1\ZZ\cap \alpha_2\ZZ=\{0\}$) then $f=\id:\RR\to\RR$ has a decomposition as $f=f_1+f_2$, $\Diff_{\alpha_j}f_j=0$.
\begin{proposition}
Let $\alpha_1,\alpha_2\in \RR$ be incommensurable. Then each function $f:\RR\to \RR$ satisfying \eqref{eq:kern} can be written as $f=f_1+f_2$, with $f_1, f_2$ being $\alpha_1$ and $\alpha_2$ periodic, respectively. That is, $\RR^\RR$ has the decomposition property with respect to any system of two incommensurable reals.
\end{proposition}
\begin{proof}
Using the \emph{axiom of choice}, we can select one
representative from each of the classes of the equivalence
$x\sim y \Leftrightarrow x-y\in \alpha_1\ZZ+\alpha_2\ZZ$.

On each class we construct our $f_j$ as follows. For the fixed
class representative $y\in\RR$ take
$f_1(y+k\alpha_1+m\alpha_2):=f(y+m\alpha_2)$ and
$f_2(y+k\alpha_1+m\alpha_2):=f(y+k\alpha_1)-f(y)$. Then $f_j$ are
$\alpha_j$-periodic and by \eqref{eq:kern}
\begin{align*}
f(y+k\alpha_1+m\alpha_2)& =f(y+m\alpha_2)+f(y+k\alpha_1)-f(y)\\& =
f_1(y+k\alpha_1+m\alpha_2)+{f_2(y+k\alpha_1+m\alpha_2)}.
\end{align*}
This ends the construction of a periodic decomposition. \qed
\end{proof}

The above decomposition can be far worse, than the function itself. E.g., $f=\id$ is continuous, while $f_1$ and $f_2$ are certainly not, for continuous periodic functions, hence also their sums, are necessarily bounded. That $f=\id$ does not even have a measurable decomposition, is proved in \cite{Natkaniec} by a somewhat involved argumentation.

In fact, no function with $\lim_{x\to\infty}f(x)=\infty$ can have a measurable periodic decomposition. Indeed, let $\varepsilon, \eta >0$ be fixed arbitrarily, and assume that $f$ has a measurable decomposition \eqref{eq:periodd}. Then for each $j=1,\dots,n$, $f_j$ must be bounded on $[0,\alpha_j]$ by some constant $K_j<\infty$ apart from an exceptional set $A_j\subseteq [0,\alpha_j]$ of measure $|A_j|<\eta$. Therefore, on any interval $I$ of length $\ell$ (large), $f$ is bounded by $K:=K_1+\dots +K_n<\infty$ apart from an exceptional set $A\subseteq I$ of measure $|A|<(\lceil \ell/\alpha_1\rceil+\dots+\lceil \ell/\alpha_n\rceil)\eta<\varepsilon \ell$, if $\eta$ is chosen small enough. So $f$ is ``locally almost bounded'': for any $\varepsilon>0$ there is $K<\infty$ such that on any sufficiently large interval $I$, $|\{x\in I~:~ |f(x)|>K\}|<\varepsilon |I|$.

One would think that the bug here is with the axiom of choice, the huge number of ``ugly'', non-measurable functions, so that once a continuous function has a relatively nice---say, measurable---decomposition, then it must also have a continuous one. However, the contrary is true:
\begin{proposition}[T.{} Keleti \cite{keleti:1997}]
There exists $f\in \Ce(\RR)$ having measurable decomposition \eqref{eq:periodd} but without a continuous periodic decomposition.
\end{proposition}
For the proof see \cite[Thm.{} 4.8]{keleti:1996}.

We can also look for further immediate solutions of \eqref{eq:kern}: For example polynomials of degree $m<n$ satisfy this difference equation. So, we can ask for \emph{quasi-decompositions} with \emph{periodic functions and polynomials}
\begin{equation} \label{eq:quasidec}
f=p+f_1+\dots+f_n,\quad\text{with}~~ \Diff_{a_j}f_j=0 \quad
\text{and}~~ \deg p <n ~~\text{a polynomial}.
\end{equation}

\begin{theorem}[I.{}Z.{} Ruzsa, M.{} Szegedy (unpublished)]
There exist continuous, {unbounded}
solutions of \eqref{eq:kern} with $\lim_{x\to\infty} f(x)/x=0$.
\end{theorem}
As a consequence $\Ce(\RR)$ does not have the quasi-decomposition property either.
 For a discussion see \cite[pp.{} 338--339]{laczkovich/revesz:1989}.  It can be precisely described which functions in $\Ce(\RR)$ have continuous periodic quasi-decompositions \eqref{eq:quasidec}.

\begin{theorem}[M.{} Laczkovich, Sz.{} R\'ev\'esz \cite{laczkovich/revesz:1989}]\label{th:LRC}
For a function $f\in \Ce(\RR)$ the existence of a
quasi-de\-com\-po\-si\-tion \eqref{eq:quasidec} is equivalent
to \eqref{eq:kern} together with {the Whitney condition}
\begin{equation*}
\delta_n(f):=\sup \Bigl\{ \sum_{j=0}^n (-1)^j \binom{n}{j}
f(x+jh)~:~~x,h\in\RR \Bigr\} <\infty.
\end{equation*}
\end{theorem}
\begin{proof} Obviously, \eqref{eq:quasidec} implies both \eqref{eq:periodd} and $\delta_n(f)<\infty$. Conversely, a result of H.{} Whitney \cite{Whitney} says that $\delta_n(f)<\infty$ entails that $f$ can be approximated by a polynomial $p$ of degree $\deg p <n$ within a bounded distance: $\|f-p\|_\infty<\infty$. Thus, for $g:=f-p\in \BC(\RR)$ we have $\Diff_{\alpha_1}\dots\Diff_{\alpha_n} g=0$ and it remains to establish the decomposition property of $\BC(\RR)$, postponed to \S\ref{sec:BCR} below.\qed
\end{proof}

\section{Continuous periodic decompositions}\label{sec:BCR}

In view of the foregoing discussion it is natural to pose the boundedness condition on the occurring functions and look at subclasses $\FC$ of the space $\Cb(\RR)$ of bounded continuous functions  on $\RR$. Note that if $f$ has a continuous periodic decomposition it is \emph{uniformly almost periodic} (alternatively, Bohr or Bochner  almost periodic), i.e., the set
\begin{equation*}
\bigl\{f(\cdot+t):t\in \RR\bigr\}\subseteq \Cb(\RR)
\end{equation*}
of its translates  is relatively compact with respect to the supremum norm
$\|f\|_\infty:=\sup_{x\in \RR}|f(x)|$.
Denote by $\AP(\RR)$ the set of all such functions, which becomes a Banach space, actually a Banach algebra, if endowed with the supremum norm and pointwise operations, see \cite[Ch.{} I.]{Besi}.
Evidently, a solution of \eqref{eq:kern} in  $\FC\subseteq\Cb(\RR)$ must be contained by $\AP(\RR)$ if $\FC$ has the decomposition property.
\begin{proposition}\label{p:AP}
The space $\AP(\RR)$ has the decomposition property.
\end{proposition}
At this point, we give a proof only for the case of incommensurable periods to illustrate the use of Fourier analytic techniques.  The complete proof will be given in \S\ref{sec:mean} as a special case of a more general result.
\begin{proof}
Suppose $\alpha_1,\dots,\alpha_n$ are incommensurable  and let $f\in \AP(\RR)$.
Any $f\in \AP(\RR)$ has a mean value
\begin{equation*}
Mf:=\lim_{t\to \infty} \frac1 t\int_0^t f(s)\dd s
\end{equation*}
by \cite[Sec.{} I.3]{Besi}, and $M$ is a continuous linear functional on $\AP(\RR)$.  Moreover, the Fourier coefficients of $f$ can be calculated
as $(0\ne )c_k:=a(\lambda_k):=M(f(s)\ee^{-is\lambda_k})$ for some uniquely determined sequence $(\lambda_k)$, and $f$ has the Fourier series $f\sim\sum_{k=1}^\infty c_k \ee^{ix\lambda_k}$.
Now for any $\alpha\in \RR$
\begin{align*}
\Diff_\alpha  \Bigl(M(f\ee^{-is\lambda_k})\ee^{ix\lambda_k}\Bigr)=M(\Diff_\alpha f(x-t) \ee^{-it\lambda_k})=M(\Diff_\alpha f(s) \ee^{-is\lambda_k})\ee^{ix\lambda_k}.
\end{align*}
So that the difference equation \eqref{eq:kern} implies $\Diff_{\alpha_1}\dots \Diff_{\alpha_n}c_k\ee^{ix\lambda_k}=0$. Since $c_k\neq 0$, this is only possible if $\lambda_k={2\pi \ell}/{\alpha_j}$ for some $\ell\in \ZZ$ and $j\in \{1,\dots,n\}$. Since $\alpha_1,\dots,\alpha_n$ are incommensurable there can be at most one such $j$. On the other hand,  by Section I.8.6$^\circ$ in \cite{Besi} $\frac1N\sum_{j=1}^N f(s+k\alpha_j)$ converges uniformly (in $s$) to a bounded continuous function $f_j$ which is $\alpha_j$-periodic and whose Fourier coefficients are precisely those Fourier coefficients $a(\lambda)$ of $f$ for which $\lambda\in (2\pi/\alpha_j)\ZZ$. We see therefore that $f_1+\cdots+f_n$ and $f$ have the same Fourier coefficients, hence they coincide by Theorem I.4.7$^\circ$ in \cite{Besi}.
\qed\end{proof}
\noindent That is to say if we \emph{a priori} know that $f$ is uniformly almost periodic, then the difference equation \eqref{eq:kern} implies the periodic decomposition \eqref{eq:periodd}.

The next step is to deduce this almost periodicity. Let $\mu\in\MMc(\RR)$, i.e., a compactly supported finite (signed) Borel measure on $\RR$, and let $f\in \Ce(\RR)$. Then
\begin{equation*}
f*\mu(x):=\int_{\RR}f(x-t)\dd\mu(t)
\end{equation*}
defines a continuous function, the \emph{convolution} of $f$ and $\mu$. The convolution of two measures $\mu,\nu\in\MMc(\RR)$ is defined by
$f*(\mu*\nu):=(f*\mu)*\nu$ (for $f\in \Ce(\RR)$): As a continuous linear functional on the locally convex space $\Ce(\RR)$, $\mu*\nu$ is a compactly supported measure, i.e., $\mu*\nu\in\MMc(\RR)$. It is also easy to see that convolution is commutative and associative in $\MMc(\RR)$.

Now denote $\mu_\alpha:=\delta_{-\alpha}-\delta_0$, where $\delta_\beta$ is the Dirac measure at $\beta\in\RR$. Then $f*\mu_\alpha=f*(\delta_{-\alpha}-\delta_0)=\Delta_\alpha f$.
With this equation \eqref{eq:kern} takes the form
\begin{equation*}
f*(\mu_{\alpha_1}*\cdots *\mu_{\alpha_n})= f*\bigl((\delta_{-\alpha_1}-\delta_0)*\cdots*(\delta_{-\alpha_n}-\delta_0)\bigr)=0.
\end{equation*}
\begin{definition}[L.{} Schwartz \cite{LSch}]\label{d:mp}
A function $f\in \Ce(\RR)$ is
\emph{mean periodic} if there exists a compactly
supported Borel measure $\mu$ on $\RR$ with $f*\mu=0$.
\end{definition}
Let us recall from  \cite[p.{} 44]{Kah} the following.
\begin{proposition}[J.-P. Kahane]\label{p:JPK}
A bounded uniformly continuous mean periodic function is uniformly almost periodic.
\end{proposition}

An immediate consequence of this and of Proposition \ref{p:AP} is the following.
\begin{proposition}[Z.{} Gajda \cite{gajda:1992}]\label{p:BUC}
$\BUC(\RR)$ has the decomposition property.
\end{proposition}

Gajda proved this results  with a different argument (using Banach limits) that can be easily extended to the case of translations on locally compact Abelian groups (see Corollary \ref{cor:agrpdecomp} below).

However, the result of Gajda for $\BUC(\RR)$ falls short of the complete truth, in the extent that it does not tell that  a continuous function satisfying \eqref{eq:kern} is necessarily uniformly continuous, a fact that would imply even the decomposition property of the whole $\Cb(\RR)$ itself.

No direct proof of the implication ``$f\in \BC(\RR)$ \& \eqref{eq:kern}
$\Rightarrow$ $f\in \BUC(\RR)$'' is known, so the decomposition property of $\BC(\RR)$ lies deeper. In fact, to prove that a bounded continuous solution of \eqref{eq:kern} is uniformly continuous, we have no other known ways than this periodic decomposition result on $\Cb(\RR)$ itself.

Before proceeding let us formulate  the following more general question than Problem \ref{p:Ruzsa}.
\begin{problem}\label{p:measuredec} Let $\mu, \nu$ 
be given Borel measures of compact support on $\RR$. Clearly, if
\begin{equation}\label{eq:measuredec}
f=g+h \quad \text{with}\quad g,h\in \Ce(\RR) \quad \text{such that} \quad g*\mu=0, ~ h*\nu=0,
\end{equation}
then $f*(\mu*\nu)=0$. Find conditions, under which we have the
converse implication: If $f\in \Ce(\RR)$, and $f*(\mu*\nu)=0$,
then \eqref{eq:measuredec} holds. Or find conditions on $\mu$ ensuring that a
solution $f\in \BC(\RR)$ of $f*\mu=0$ is almost periodic.
\end{problem}
In this formulation we use no assumption on boundedness or
uniform continuity. Clearly, then additional assumptions are
needed. E.g.{} additional functional equations must also be
satisfied? Spectra must be simple? Spectra of $\mu$ and $\nu$
should be distinct? Several variations may be considered.
\begin{remark}
In the problem above $f$ is by default mean periodic. However, convergence of mean periodic Fourier
expansions was shown only in a complicated, complex sense. Perhaps, recent developments in the Fourier synthesis and representation of mean periodic functions  can be used, see Sz\'ekelyhidi \cite{Sz}. Then again, boundedness and uniform continuity could be of use by means of Proposition \ref{p:JPK} of Kahane.
\end{remark}

M.{} Wierdl \cite{wierdl:1984} showed that $\BC(\RR)$ has the  $2$-decomposition property. Subsequently, Laczkovich and  R\'{e}v\'{e}sz proved this for general $n$ as the main result of \cite{laczkovich/revesz:1989}, which was the first internationally published paper in this topic (but see also the preceding paper \cite{ck30}).

Although many generalizations and interpretations have since been described and various tools could be invoked depending on the setup, oddly enough this first non-trivial result could be covered by neither extensions. To date, we have no other proof than the essentially elementary yet tricky original argument, which we will describe also here below.

\begin{theorem}[M.{} Laczkovich, Sz.{} R\'ev\'esz \cite{laczkovich/revesz:1989}]\label{th:LRBC}
The Banach space $\Cb(\RR)$ has the decomposition property.
\end{theorem}

We devote the rest of this section to the proof of Theorem \ref{th:LRBC}. We slightly differ from the original proof of \cite{laczkovich/revesz:1989}, in  exploiting Proposition \ref{p:AP}.

\smallskip\noindent  For $n=1$ the statement is trivial, so we argue by induction. Suppose $f\in \Ce(\RR)$ satisfies
\eqref{eq:kern}.  We group the steps according to commensurability:
\begin{equation*}\{\alpha_1,\dots,\alpha_n\}=\{\alpha_1,\dots,\alpha_a\}\cup\{\beta_1,\dots,\beta_b\}
\cup\dots \{\rho_1,\dots,\rho_r\}.\end{equation*} Denote the \emph{least common multiple} of these by  $\alpha,\beta,\dots,\rho$, i.e., $\alpha$ is the non-negative generator of the cyclic group $\bigcap_{j=1}^a \alpha_j\ZZ$ etc. Then from \eqref{eq:kern} we obtain
\begin{equation}\label{eq:commengrouped}
\Diff_{\alpha}^a\dots\Diff_{\rho}^r f=0.
\end{equation}

\begin{lemma}\label{l:commensurable} Let $f\in \Bb(\RR)$
and $\alpha\in\RR$, $n\in\NN$. If $\Diff_{\alpha}^nf=0$, then
$\Diff_{\alpha}f=0$.
\end{lemma}
\begin{proof}
Obviously, it suffices to work out the proof for  $n=2$.
Let $g:=\Diff_{\alpha} f$. By condition, $\Diff_{\alpha} g=0$, so $g$ is $\alpha$-periodic. Therefore,
\begin{equation*}
f(x+N\alpha)=f(x)+\sum_{i=0}^{N-1} \Diff_{\alpha} f(x+i\alpha)=f(x)+Ng(x),
\end{equation*}
thus $ \|g\|_{\infty} \leq \frac{2}{N} \|f\|_{\infty} \to 0$ ($N\to \infty$).
That is, $g:=\Diff_\alpha f=0$, as needed. \qed\end{proof}
As a consequence, from \eqref{eq:commengrouped} we obtain
\begin{equation}\label{eq:commreduced}
\Diff_{\alpha}\dots\Diff_{\rho} f=0.
\end{equation}
Hence in case $\alpha_1,\dots, \alpha_n$ are not all pairwise incommensurable then $f$ is also a solution of a difference equation of order less than $n$. We can therefore apply the induction hypothesis providing that $f$ has an $(\alpha,\dots, \rho)$-decomposition. So in particular $f\in \AP(\RR)$, which space has the decomposition property in view of Proposition \ref{p:AP}, and so we are done.

\smallskip\noindent It remains to handle the case when $\alpha_1,\dots, \alpha_n$ are pairwise incommensurable. The crux of the proof is thus the following:
\begin{lemma}\label{l:incommensurable} Let $\alpha_1,\dots,\alpha_n$ be pairwise
incommensurable, and let $f\in \BC(\RR)$ satisfy
\eqref{eq:kern}. Then $f$ has an
$(\alpha_1,\dots,\alpha_n)$-decomposition in $\BC(\RR)$.
\end{lemma}
To prove this lemma it is natural to get rid of one period and reduce the situation to a difference equation of order $n-1$ by considering
$
g:=\Diff_{\alpha_n} f,~~
$
which then satisfies $\Diff_{\alpha_1}\dots\Diff_{\alpha_{n-1}} g=0$,
and thus by the induction hypothesis
\begin{equation*}
g=g_1+\dots+g_{n-1}\qquad (\Diff_{\alpha_j} g_j=0,~j=1,\dots,n-1).
\end{equation*}
If $f$ were subject to the representation \eqref{eq:periodd},
then we could guess $\Diff_{\alpha_{n}} f_j=g_j$. So we try to
``lift up'' the $g_j$ to some functions $f_j$ with $\Diff_{\alpha_j}
f_j=\Diff_{\alpha_j} g_j=0$ and $\Diff_{\alpha_n}f_j=g_j$.
Suppose this works, we find such $f_j\in \BC(\RR)$. Then
\begin{equation*}
f_n:=f-(f_1+\dots+f_{n-1}) \in \BC(\RR),
\end{equation*}
and $\Diff_{\alpha_n}f_n=g-(g_1+\dots+g_{n-1})=0$, so $f$ has a
decomposition \eqref{eq:periodd}.
So it remains to see the possibility of a lift-up for any incommensurable periods.
\begin{lemma}\label{l:liftequiv} Let $g\in \Ce(\RR)$, let
$\beta,\gamma\in\RR$ be incommensurable, and suppose
$\Diff_{\beta} g =0$. Then the following are equivalent:
\begin{enumerate}[(i)]
\item There exists $K>0$ such that
\begin{equation*}\Bigl| \sum_{j=0}^{k-1}
    g(x+j\beta)\Bigr|<K \quad\mbox{(for $x\in\RR, k\in \NN$)}.\end{equation*}

\item There is $h\in \Ce(\RR)$ such that $\Diff_{\beta}
    h=0$ and $\Diff_{\gamma} h=g$.
\end{enumerate}
\end{lemma}
\begin{proof} This is a special case of a well-known ergodic theory
result, see \cite[Thm.{} 14.11, p.{} 135]{GH}, as putting $Y:=\RR/\gamma\ZZ$, the homeomorphism
$\Theta(x):=x+\beta \mod \gamma$ has minimal orbit-closure $Y$
for every $x$.
\qed\end{proof}

\section{Generalizations to linear operators}\label{sec:mean}
For $\alpha\in \RR$ the translation by $\alpha$ acts as a homeomorphism on $\RR$. Consider the so-called \emph{Koopman  (or composition) operator}, in this case called the \emph{shift operator},
\begin{equation*}
T_{\alpha}:\RR^\RR\to \RR^\RR, \quad T_{\alpha}f(x):=f(x+\alpha).
\end{equation*}
Observe that the solutions of the difference equation \eqref{eq:kern} form the subspace
\begin{equation*}
\ker(T_{\alpha_1}-\Id)\cdots (T_{\alpha_n}-\Id)
\end{equation*}
while the functions having a periodic decomposition \eqref{eq:periodd}
are the elements of
\begin{equation*}
\ker(T_{\alpha_1}-\Id)+\cdots+\ker(T_{\alpha_n}-\Id).
\end{equation*}
Then  Problem \ref{p:Ruzsa} can be rephrased so as whether the equality
\begin{equation}
\ker(T_{\alpha_1}-\Id)\cdots (T_{\alpha_n}-\Id)=\ker(T_{\alpha_1}-\Id)+\cdots+\ker(T_{\alpha_n}-\Id)
\end{equation}
holds?  Of course, one can restrict the question by considering linear subspaces of $\RR^\RR$ that are invariant under the occurring operators. The equality then means the decomposition property of $\FC$. Or more generally one can ask the following:
\begin{problem}\label{p:kereqT}
Let $E$ be a linear space and let $T_1,\dots, T_n:E\to E$ be commuting linear operators. Find conditions such that
\begin{equation}\label{eq:kereqT}
\ker(T_1-\Id)\cdots( T_n-\Id)=\ker(T_1-\Id)+\cdots+\ker(T_{n}-\Id).
\end{equation}
\end{problem}
\begin{remark}
For a system of pairwise commuting operators $T_1,\dots, T_n$ the inclusion
``$\ker(T_1-\Id)\cdots( T_n-\Id)\supseteq\ker(T_1-\Id)+\cdots+\ker(T_{n}-\Id)$''
trivially holds. This corresponds to the trivial implication ``\eqref{eq:periodd} $\Rightarrow$ \eqref{eq:kern}''.
\end{remark}
We start with a model result. Let $E$ be a Banach space and denote by $\LLL(E)$ the space of bounded linear operators on $E$. Recall that $T\in\LLL(E)$  is \emph{mean ergodic} if its Ces\`aro means
\begin{equation*}
\frac{1}N\sum_{j=1}^{N} T^j
\end{equation*}
converge in the strong operator topology, i.e., pointwise for $x\in E$. In this case the limit $P$ is the so-called \emph{mean ergodic projection} onto $\ker(T-\Id)$ and one has $E=\ran P\oplus \ker P$ and $\ker P=\clran(T-\Id)$, where $\ran(T)$ denotes the range of the operator $T$, see \cite[Sec.{} 8.4]{EFHN}.

\begin{proposition}\label{p:meanerg}
Let $E$ be a Banach space and $T_1,\dots, T_n\in \LLL(E)$ be commuting
mean ergodic operators with. Then the equality \eqref{eq:kereqT} holds.
\end{proposition}
\begin{proof}
Since the operators $T_1, \dots, T_n$ commute, so do the mean ergodic projections $P_1,\dots, P_n$, and actually all occurring operators commute with each other. A moment's thought explains that the \emph{direct} decomposition
\begin{align*}
E&=\ran P_1\oplus \ran P_2(\Id-P_1)\oplus \cdots \oplus \ran\bigl( P_n(\Id-P_{n-1})\cdots(\Id-P_1)\bigr)\\
&\quad\oplus \ran\bigl( (\Id-P_n)(\Id-P_{n-1})\cdots(\Id-P_1)\bigr)
\end{align*}
is valid, i.e. for any $x\in E$ we can write $x=x_1+\cdots +x_n+y$ with $x_i\in \ran P_i = \ker(T_i-\Id)$ and $y\in\ran(\Id-P_1)\cdots (\Id-P_n)$. Let now $x\in \ker (T_1-\Id)\cdots (T_n-\Id)$: then $(T_1-\Id)\cdots (T_n-\Id)y=0$. It follows that $y\in \ker (T_1-\Id)\cdots (T_n-\Id)\subseteq \ker (\Id-P_1)\cdots (\Id-P_n)$, thus $y\in  \ran(\Id-P_1)\cdots (\Id-P_n) \cap \ker (\Id-P_1)\cdots (\Id-P_n)$. However, $(\Id-P_1)\cdots (\Id-P_n)$ is a projection, so from this $y=0$ follows.
\qed\end{proof}
Actually, the proof above and the result itself appears in \cite{KadetsAve} in a slightly more general form, and as a matter of fact even much earlier in \cite{laczkovich/revesz:1990}. None of the papers however formulated it by using the notion of mean ergodicity.
\begin{example}
Since shift operators $T_\alpha$ are all mean ergodic on $E=\AP(\RR)$ we obtain another proof of Proposition \ref{p:AP}. To see that $T_\alpha$ is mean ergodic it suffices to note that $\{T_\alpha^n:n\in \NN\}$ is compact in the strong operator topology and to invoke \cite[Thm.{} 8.20]{EFHN}; or alternatively one can use \cite[Sec.{} I.8.6$^\circ$]{Besi} as in the proof of Proposition \ref{p:AP}.
\end{example}
\begin{remark}
\begin{enumerate}[a)]
\item One trivially has $\|P_j\|\leq \|T_j\|$ and $\|\Id-P_j\|\leq 1+\|T_j\|$. Suppose $\|T_j\|\leq 1$ for $j=1,\dots, n$. The proof above yields that the decomposition obtained is actually
\begin{equation*}
x=P_1x+P_2(\Id-P_1)x+\cdots+P_n(\Id-P_{n-1})(\Id-P_2)(\Id-P_1)x.
\end{equation*}
Hence $x$ has a decomposition $x=x_1+\cdots+x_n$ with $x_j\in \ker(T_j-\Id)$ and
\begin{equation*}
\max_{j=1,\dots, n}\|x_n\|\leq2^{n-1}\|x\|.
\end{equation*}
\item If $E$ is a Hilbert space, then the mean ergodic projections $P_j$ are orthogonal, see \cite[Thm.{} 8.6]{EFHN}.  So that $\Id-P_j$ is also an orthogonal, hence contractive, projection. This implies that $x\in \ker (T_1-\Id)\cdots (T_n-\Id)$  has a decomposition $x=x_1+\cdots+x_n$ with $x_j\in \ker(T_j-\Id)$ and
\begin{equation*}
\max_{j=1,\dots, n}\|x_n\|\leq\|x\|.
\end{equation*}
\item In the original setting of the decomposition problem Laczkovich and R\'ev\'esz have shown that on $E=\Cb(\RR)$ with $T_j$ being translations by $a_j$ a function $f$ satisfying \eqref{eq:kern} has a decomposition $f=f_1+\cdots+ f_n$ with
\begin{equation*}
\max_{j=1,\dots, n}\|f_j\|_\infty \leq 2^{n-2}\|f\|.
\end{equation*}
The estimate is sharp for $n=2$, see \cite{laczkovich/revesz:1989}.
\end{enumerate}
\end{remark}
\begin{problem}
Find the best constant $C_n$ such that any $x\in \ker (T_1-\Id)\cdots (T_n-\Id)$ has some decomposition $x=x_1+\cdots+x_n$ with $x_j\in \ker(T_j-\Id)$ and
\begin{equation*}
\max_{j=1,\dots, n}\|x_n\|\leq C_n \|x\|.
\end{equation*}
We saw $C_n\leq 2^{n-1}$ in general, $C_n\leq 2^{n-2}$ for translations on $\Cb(\RR)$. Are these estimates sharp? Is it true that $C_n=1$ for translations on $\Cb(\RR)$ for every $n$?
Under which conditions on $E$ and/or $T_1,\cdots, T_n$ does $C_n=1$ hold?
\end{problem}
\begin{example}\label{ex:refl}
It is a classical result that a power bounded operator on a reflexive Banach space $E$ is mean ergodic. As a consequence, commuting power bounded operators on a reflexive Banach space $E$  fulfill the conditions of Proposition \ref{p:meanerg}, hence \eqref{eq:kereqT} holds true. See also \cite[Cor.{} 2.6]{laczkovich/revesz:1990}
\end{example}
\begin{theorem}[M.{} Laczkovich and Sz.{} R\'ev\'esz  \cite{laczkovich/revesz:1990}]\label{th:topvsp} Let $X$ be a topological vector
space and $T_1,\dots,T_n$ be commuting continuous linear operators on $X$. Suppose that for every $x\in X$ and $j\in \{1,\dots,n\}$ the closed convex hull of $\{T_j^mx:m\in\NN\}$ contains a fixed point of $T_j$, that is
\begin{equation*}
\clconv\{T_j^mx:m\in\NN\}\cap \ker(T_j-\Id)\neq \emptyset.
\end{equation*}
Then \eqref{eq:kereqT} holds.
\end{theorem}
The crux of the proof is same as for Proposition \ref{p:meanerg}.
Instead of proving this theorem (for the proof see \cite{laczkovich/revesz:1990}), we only remark that if $X=E$ is a Banach space and $T_1,\dots, T_n$ are power bounded, the fixed point condition in Theorem \ref{th:topvsp} means precisely the mean ergodicity  of $T_1,\dots,T_n$, see \cite[Theorem 8.20]{EFHN}.

\begin{corollary}\label{c:vectortopology} Let $X$ be a Banach space and let $T_1,\dots,T_n \in \LLL(X)$ be commuting power bounded operators. Suppose an additional vector topology $\tau$ is given on $E$ such that the unit ball $\Ball:=\{x\in X:\|x\|\leq 1\}$ is $\tau$-compact, and the operators $T_j$ are $\tau$-continuous. Then \eqref{eq:kereqT} holds.
\end{corollary}
The proof is the application of the foregoing result and the Markov--Kakutani fixed point theorem (see, e.g., \cite[Sec.{} 10.1]{EFHN}) to the closed convex hull $\clconv\{T_j^mx:m\in \NN\}$, which was assumed to be $\tau$-compact.

The above together with the Banach--Alaoglu theorem yields the following:
\begin{proposition}\label{p:weakstar}
Let $X$ be a normed space, $E=X^*$ and let $\tau=\sigma(X^*,X)$ be the weak$^*$ topology on $E^*$. If $T_1,\dots, T_n\in \LLL(E)$ are commuting, power bounded weakly$^*$ continuous operators, then \eqref{eq:kereqT} holds.
\end{proposition}
\begin{definition}
Let $E$ be a Banach space, or, more generally, a topological vector space.
We say that $E$ has the \emph{decomposition property with respect to the pairwise commuting operators} $T_1,\dots, T_n\in \LLL(E)$ if \eqref{eq:kereqT} holds.
Moreover, if $\mathscr{A}\subseteq \LLL(E)$ and $E$ has the decomposition property for each system of $n$ pairwise commuting operators $T_1,\dots, T_n \in \mathscr{A}$, then $E$ is said to have the \emph{$n$-decomposition property with respect to $\mathscr{A}$}.
Finally, if this holds for all $n\in \NN$, then $E$ is said to have the \emph{decomposition  property with respect to $\mathscr{A}$}.
\end{definition}
So that e.g. Example \ref{ex:refl} means that a reflexive Banach space has the decomposition property with respect to (commuting) power bounded operators. This new terminology shall not cause any ambiguity in connection with the decomposition property of function classes $\mathscr{F}\subseteq \RR^\RR$ (in Definition \ref{d:dprall}).

\begin{remark}
If $1$ is not an eigenvalue of say $T_1$, then the questioned  equality \eqref{eq:kereqT}  reduces to
$\ker (T_2-\Id)\dots (T_n-\Id)=\ker (T_2-\Id)+\cdots+\ker (T_n-\Id)$. That is to say the order $n$ reduces to order $n-1$. In particular, if $1$ is not a spectral value for every $T_1,\dots, T_n$, then \eqref{eq:kereqT} is satisfied trivially, both sides being $\{0\}$.
\end{remark}
Note the following border-line feature of our subject matter.  It is only interesting to look at cases when  $\|T_1\|\geq1$, $\dots$, $\|T_n\|\geq 1$ (since $\Id-T$ is invertible for $\|T\|<1$). On the other hand, if $T_1,\dots, T_n$ are power bounded and commute, we can equivalently renorm $E$ by
$\nm x\nm:=\sup_{k_1,\dots, k_n\in \NN}\bigl\|T_1^{k_1}\dots T_n^{k_n}x\bigr\|,~
$
such that for the new norm each operator becomes a contraction. Hence in the end with the assumption  $\|T_1\|=\cdots =\|T_n\|=1$ one loses no generality (for the particularly fixed  power bounded operators $T_1,\dots, T_n$).

\smallskip

Recall that a Banach space $E$ is called \emph{$m$-quasi-reflexive} if  $E$ has codimension $m$ in its bidual $E^{**}$.
\begin{theorem}[V.{}M.{} Kadets, S.{}B.{} Shumyatskiy \cite{kadets/shumyatskiy:2001}]
\begin{enumerate}[a)]
\item A $1$-quasi reflexive  Banach space  $E$ has the $2$-decomposition property with respect to any
pair of commuting linear transformations $S, T$ of norm $1$.
\item  If $E$ is $m$-quasi reflexive with  $m>1$, then there exist commuting linear
transformations $S, T\in \LLL(E)$ of norm $1$ such that $E$ fails to have the
$2$-decomposition property with respect to $S, T$.
\end{enumerate}
\end{theorem}
Also Kadets and Shumyatskiy proved the following:

\begin{theorem}[V.{}M.{} Kadets, S.{}B.{} Shumyatskiy \cite{KadetsAve}]
Neither the space $\co$ of null sequences, nor $\ell^{1}$ has the $2$-decomposition property with respect to operators of norm $1$.
\end{theorem}
See \cite{KadetsAve} for the proofs and for further information on averaging techniques which can be used in connection with the periodic decomposition problem. Several natural questions arise, see \cite{kadets/shumyatskiy:2001}:
\begin{problem}\label{p:quasiref}
\begin{enumerate}[1.]
\item Is it true that in a $1$-quasi reflexive space $E$ has the decomposition property with respect to any finite system of commuting operators of norm $1$?
\item Does the  $2$-decomposition property with respect to contractions imply the $n$-decomposition property with respect to contractions?
\item Does the $2$-decomposition property with respect to power bounded operators characterizes $m$-quasi reflexive Banach spaces with $m\leq 1$?
\end{enumerate}
\end{problem}
Let us finally remark that a recent result of V.P.{} Fonf, M.{} Lin and P.{} Wojtaszcyk \cite{fonf} states that a separable $1$-quasi reflexive space can be equivalently renormed such that every contraction with respect to the new norm becomes mean ergodic. Also a classical result of theirs, see \cite{Fonf0}, is that a Banach space $E$ is reflexive if (and only if) every power bounded operator is mean ergodic. These indicate the possible difficulty of Problem \ref{p:quasiref}.
\subsubsection{Applications to $\Ell^p$ spaces}
We first discuss immediate consequences of the previous operator-theoretic results.
Let $(X,\Sigma,\mu)$ be a a measure space. In this section our standing assumption is as follows:

\begin{condition}\label{con:measless}
For $j=1,\dots,n$ let $T_j:X\to X $ be pairwise commuting measurable mappings such that
$\mu(T_j^{-1}(A))\leq \mu(A)$ for every $A\in \Sigma$.
\end{condition}
Then the Koopman operators, denoted by the same letter and defined by
\begin{equation*}
T_jf:=f\circ T_j
\end{equation*}
are contractions on all of the spaces $\Ell^p(X,\Sigma,\mu)$. In particular the condition above is fulfilled if the $T_j$ are measure-preserving, in which case the Koopman operators $T_j$ become isometries on each of the $\Ell^p$ spaces.

For the reflexive range the next corollary of Proposition \ref{p:meanerg} is immediate:
\begin{corollary}
Let $1<p<\infty$.
Under Condition \ref{con:measless} consider the Koopman operators $T_j$ on $\Ell^p(X,\Sigma,\mu)$. Then \eqref{eq:kereqT} holds true.
\end{corollary}

The same result is true for the case $p=1$, but the proof is different since infinite dimensional $\Ell^1$ spaces are non-reflexive.
We remark however that if $(X,\Sigma,\mu)$ is finite, then the Koopman operators $T_j$ are simultaneous $\Ell^1$ and $\Ell^\infty$ contractions, so-called Dunford--Schwartz operators, that are known to be mean ergodic on $\Ell^1$, see, e.g.,  \cite[Sec.{} 8.4]{EFHN}.
\begin{proposition}[M.{} Laczkovich, Sz.{} R\'ev\'esz \cite{laczkovich/revesz:1990}]
Under Condition \ref{con:measless} consider the Koopman operators $T_j$ on $\Ell^1(X,\Sigma,\mu)$. Then \eqref{eq:kereqT} holds true.
\end{proposition}
We do not give the proof here, but note that the mean ergodicity of the operators  can be replaced by an application of Birkhoff's pointwise ergodic theorem, see, e.g.{} \cite[Ch.{} 11]{EFHN}. See \cite{laczkovich/revesz:1990} for the detailed proof.

The case of $p=\infty$ is more subtle. Let us recall the following notion.
\begin{definition}
A measure space $(X,\Sigma,\mu)$ is called \emph{localizable} if the dual of the Banach space $\Ell^1(X,\Sigma,\mu)$ is $\Ell^\infty(X,\Sigma,\mu)$ (with the usual identification).
\end{definition}
As a matter of fact, the original definition of Segal (see \cite[Sec.{} 5]{segal}) was different, but is equivalent to the  one above. Known examples of localizable measure spaces include:
\begin{example}\label{exa:loc}
\begin{enumerate}[1.]
\item $\sigma$-finite measure spaces,
\item $(X,\Sigma,\mu)$ with $X$ a set $\Sigma=\pow(X)$ the power set, $\mu$ the counting measure,
\item $(X,\Sigma,\mu)$ purely atomic,
\item $(X,\Sigma,\mu)$, $X$ a locally compact  group, $\Sigma$ the Baire algebra, $\mu$ a (left/right) Haar measure.
\end{enumerate}
\end{example}
Hence, in all of these cases the results below apply. In particular if one considers commuting left- (or right) translations on some locally compact group $G$, then the respective Koopman operators will satisfy \eqref{eq:kereqT}. Note that the left and the right Haar measures are absolutely continuous with respect to each other, so we can fix any of them for our considerations below.

\begin{theorem}[M.{} Laczkovich, Sz.{} R\'ev\'esz \cite{laczkovich/revesz:1990}]\label{thm:loca}
Let $(X,\Sigma,\mu)$ be a localizable measure space, and suppose that for the pairwise commuting measurable mappings $T_j:X\to X$ ($j=1,\dots,n$) the push-forward measures $\mu\circ T^{-1}$ are all absolutely continuous with respect to $\mu$. Then for the Koopman operators $T_j$ on $\Ell^\infty(X,\Sigma,\mu)$ \eqref{eq:kereqT} holds true.
\end{theorem}
The proof relies on the fact that under the conditions of localizability of $(X,\Sigma,\mu)$ and absolute continuity of the push-forward measures, the operators $T_j$ will be weak$^*$ continuous on $\Ell^\infty(X,\Sigma,\mu)$ hence one can apply Proposition \ref{p:weakstar}. For the details see \cite{laczkovich/revesz:1990}.
\begin{problem}
Can one drop the localizability assumption?
\end{problem}
\begin{corollary}[Z.{} Gajda \cite{gajda:1992}, M.{} Laczkovich, Sz.{} R\'ev\'esz \cite{laczkovich/revesz:1990}]
The space $\Bb(X)$ of bounded functions on a set $X$ has the decomposition property with respect to any system of commuting Koopman operators.
\end{corollary}
This follows from Theorem \ref{thm:loca} and from  Example \ref{exa:loc}.2 above. The proof of Gajda uses Banach limits, see also \S\ref{sec:sgrpact} below. Let us collect the previous results in  a final corollary:
\begin{corollary}
The Banach spaces $\Ell^p(\RR)$ ($1\leq p\leq \infty$,  Lebesgue measure) have the decomposition property.
\end{corollary}
Of course, $0$ is the one single periodic function in $\Ell^p(\RR)$ if $p<\infty$, hence the message of the previous result is that \eqref{eq:kern} has $0$ as the only  $\Ell^p$-solution if $p<\infty$. This follows also from a more general result of G.{}A.{} Edgar and J.{}M.{} Rosenblatt \cite[Cor.{} 2.7]{edgar} stating that the translates of a function  $f\in \Ell^p(\RR^d)$, $p<2d/(d-1)$ are linearly independent.

\subsubsection{More spaces with the decomposition property}

\begin{proposition}[M.{} Laczkovich, Sz.{} R\'ev\'esz \cite{laczkovich/revesz:1990}] The following spaces of real-valued functions on $\RR$ have the decomposition property:
\begin{enumerate}[a)]
\item $\BVb^1(\RR):=\bigl\{f:f\in\Bb(\RR)\text{ with unif. bdd. variation on $[x,x+1]$, $x\in\RR$}\bigr\}$
\item $\Lipb(\RR):=\bigl\{f:f\text{ is bounded and Lipschitz continuous}\bigr\}$
\item $\Lipb^{k}(\RR):=\bigl\{f:f\in \Cb(\RR)\text{ $k$ times differentiable with $f^{(k)}$ Lipschitz}\bigr\}$
\end{enumerate}
\end{proposition}
The cases a) and b) can be handled by introducing an appropriate norm turning the spaces under consideration into  Banach spaces, then by noting that the unit ball is compact for the pointwise topology. Hence Theorem \ref{th:topvsp} is applicable. Details are in \cite{laczkovich/revesz:1990}. Part c) relies on he following result interesting in its own right:
\begin{proposition}[M.{} Laczkovich, Sz.{} R\'ev\'esz {\cite{laczkovich/revesz:1990}}]
Let $\FC\subseteq\Ce(\RR)$ be a function class with the property that whenever $f\in \FC$ and $c\in\RR$ then $f+c\in \FC$. Let $k\in \NN$ and define
\begin{equation*}
\GFC:=\bigl\{f:\text{$f\in \Cb(\RR)$ is $k$-times differentiable with $f^{(k)}\in \FC$}\bigr\}.
\end{equation*}
If the function class $\FC$ has the decomposition property so does  $\GFC$.
\end{proposition}

\begin{problem}
There are several interesting Banach function spaces. Which of them do have the decomposition property? Just take your favorite  non-reflexive translation invariant Banach function space on $\RR$. Does it have the decomposition property? Denote by $\Ell^1_p(\RR)$ the set of functions with
\begin{equation*}
\|f\|_{1,p}:=\sup_{x\in\RR}\Big(\int_x^{1+x}|f(t)|^p\dd t\Bigr)^{1/p}<\infty,
\end{equation*}
and by $\SSS^p(\RR)$ the closure of trigonometric polynomials in this norm. The elements of $\SSS^p(\RR)$ are called \emph{Stepanov almost periodic functions}, see \cite{Besi}. Does the Banach space $\Ell^1_p(\RR)$  have the decomposition property? If the answer were affirmative it would follow that $f\in \Ell^1_p(\RR)$ and \eqref{eq:kern} imply that $f\in\SSS^p(\RR)$. (This is because periodic functions belong to $\SSS^p(\RR)$.) So, is an  $\Ell^1_p(\RR)$ solution of \eqref{eq:kern} Stepanov almost periodic?
\end{problem}

\subsubsection{One-parameter semigroups}
The original setting of the decomposition problem has a special feature, namely that the translation operators $T_t$ on translation invariant subspaces $E$ of $\RR^\RR$ form a one-parameter (semi)group of linear operators. In this section we shall study this aspect from a more general point of view. Given a Banach space $E$, a \emph{one-parameter semigroup} $T$ is a unital semigroup homomorphism $T:[0,\infty)\to \LLL(E)$, i.e., $T(t+s)=T(t)T(s)$ and $T(0)=\Id$ are fulfilled for every $t,s\geq 0$.  Whereas a \emph{one-parameter group} defined analogously as group homomorphism (into the group of invertible operators). On $\Bb(\RR)$ one can define the translation group by $T(t)f(x)=f(t+x)$ which is then, as said above, a one-parameter group.
\begin{problem}
Under which conditions does a Banach space $E$ have the decomposition property with respect to operators $T_1,\dots, T_n$ coming from a one-parameter (semi)group $T$ as $T_j=T(t_j)$ for some $t_j>0$, $j=1,\dots,n$?
\end{problem}
A one-parameter (semi)group is called a \emph{$C_0$-(semi)group} if it is strongly continuous, i.e., continuous into $\LLL(E)$ endowed with the strong (i.e, pointwise) operator topology. The translation group is not strongly continuous on $\Bb(\RR)$ or on $\Cb(\RR)$, but it is strongly continuous on $\BUC(\RR)$. A one-parameter (semi)group is called bounded if $\|T(t)\|\leq M$ for all $t \in [0,\infty)$ (or $t\in \RR)$. See \cite{EN00} for the general theory.
\begin{theorem}[V.{}M.{} Kadets, S.{}B.{} Shumyatskiy  \cite{kadets/shumyatskiy:2001}]\label{t:KS}
Let $T$ be a bounded $C_0$-group, and let $t_1,t_2>0$. Then
\begin{equation}\label{eq:semigroup2dpr}
\ker (T(t_1)-\Id)(T(t_2)-\Id)=\ker (T(t_1)-\Id)+\ker(T(t_2)-\Id).
\end{equation}
\end{theorem}
Translations on $\BUC(\RR)$ is a $C_0$-group of isometries, providing 
another proof of the $2$-decomposition property of $\BUC(\RR)$, formulated in Proposition \ref{p:BUC}.

In general the  idea is to find a closed subspace $F\subseteq E$ invariant under the semigroup operators $T(t)$, such that one can apply  Proposition \ref{p:meanerg} to the restricted operators. Concerning the nature of the problem there is one immediate candidate for this subspace. In what follows $T$ will be a fixed bounded $C_0$-semigroup. A vector $x\in E$ is called \emph{asymptotically almost periodic} (with respect to the semigroup $T$) if the orbit
$\{T(t) x:t\geq 0\}$ is relatively compact in $E$. Denote by $E_{\aap}$the collection of asymptotically almost periodic vectors, which is easily seen  to be a closed subspace of $E$ invariant under the semigroup operators. It can be proved that if $T$ is a bounded $C_0$-group then  for $x\in E_\aap$ one actually has also the relative compactness of the entire orbit $\{T(t)x:t\in \RR\}$. The proof of Theorem \ref{t:KS} by Kadets and Shumyatskiy establishes actually the fact that $\ker(T(t_1)-\Id)(T(t_2)-\Id)\subseteq E_\aap$.

The only known extensions/variations of the Kadets--Shumyatskiy result follow the same strategy (or some modifications of it) and are the following:
\begin{theorem}[B.{} Farkas \cite{Fa12a}]\label{thm:nocodecomp}
Let $E$ be a Banach space and let $T$ be a bounded $C_0$-group.
Suppose that $E$  does not contain an isomorphic copy of the Banach space  $\co$ of null sequences.
Then for every $n\in \NN$ and $t_1,\dots,t_n\in \RR$ we have
\begin{equation}\label{eq:semigroupdpr}
\ker(T(t_1)-I)\cdots(T(t_n)-I)=\ker(T(t_1)-I)+\cdots+\ker(T(t_n)-I).
\end{equation}
\end{theorem}
It is not surprising that Bohl--Bohr--Kadets type theorems (see \cite{KadecMI} and \cite{Bolis71}) play an important role here.
In this regard let us mention just the following:
\begin{theorem}[B.{} Basit \cite{Bolis71}, B.{} Farkas \cite{Fa12b}]\label{t:cochar}
A separable Banach space $E$ does not contain an isomorphic copy of $\co$ if and only if  for every  $x\in E$, $T\in \LLL(E)$ invertible with $T$ and $T^{-1}$ both power bounded the following statements are equivalent:
\begin{enumerate}[(i)]
 \item $\bigl\{T^{n+1}x-T^nx:n\in \NN\bigr\}$ is relatively compact.
  \item $\bigl\{T^{n+m}x-T^nx:n\in \NN\bigr\}$ is relatively compact for some $m\in \NN$, $m\geq 1$.
  \item $\bigl\{T^{n+m}x-T^nx:n\in \NN\bigr\}$ is relatively compact for all $m\in \NN$.
   \item $\bigl\{T^{n}x:n\in \NN\bigr\}$ is relatively compact.
   \end{enumerate}
   \end{theorem}
The next class of $C_0$-semigroups for which the decomposition problem has positive solution is of those that are \emph{norm-continuous at infinity}, including also \emph{eventually norm-continuous} semigroups,  see \cite{mazon} or \cite[Sec.{} II.4]{EN00} for these notions.
\begin{theorem}[B.{} Farkas \cite{Fa12a}]\label{t:normcontdecomp}
Let $T$ be a bounded $C_0$-semigroup that is norm-continuous at infinity. Then for all $n\in\NN$ and $t_1,\dots,t_n\geq0$ \eqref{eq:semigroupdpr} holds.
\end{theorem}

\begin{problem}
\begin{enumerate}
\item Is the Kadets--Shumyatskiy theorem true for every $n$? That is can one drop the geometric assumptions on $E$ from Theorem \ref{thm:nocodecomp}?
\item What about the case of $C_0$-semigroups? Can one get rid of the eventual norm-continuity in Theorem \ref{t:normcontdecomp}?
\item None of the above covers the decomposition property of $\Cb(\RR)$. What can be said about one-parameter semigroups that are only strongly continuous with respect to some weaker topology  on the Banach space $E$? Can one cover the decomposition property of $\Cb(\RR)$ by some extension of the results for one-parameter semigroups?
\end{enumerate}
\end{problem}
\section{Results for arbitrary transformations}
Let $X$ be a non-empty set.
The decomposition problem can be formulated in \emph{the whole
space of functions} $\RR^X$ with respect to \emph{arbitrary
commuting transformations} in $X^X$. To do that  to  a self map $T:X\to X$, called \emph{transformation}, we associate
the Koopman operator (denoted by the same letter) $Tf:=f\circ T$, and the $T$-\emph{difference operator}
$ \Diff_Tf:=Tf-f$.
A function $f$ satisfying $\Diff_T f\eqegy 0$ is then called
\emph{$T$-invariant}. A \emph{$(T_1,\dots,T_n)$-invariant decomposition} of some
function $f$ is a representation
\begin{equation}\label{invaridecdef}
f=f_1+\cdots +f_n ~,\qquad\text{where}\qquad  \Diff_{T_j} f_j \eqegy0 \quad (j=1,\dots,n).
\end{equation}

For pairwise commuting transformations $T_i$ the functional
equation
\begin{equation}\label{eqdifference}
\Diff_{T_1}\dots \Diff_{T_n} f = 0
\end{equation}
is evidently necessary for the existence of invariant decompositions. On the example of translations on $\RR$ we saw that it is not sufficient.
Now in this general setting our basic question sounds:
\begin{problem}
Give necessary and sufficient conditions, containing \eqref{eqdifference}, in order to have some $(T_1,\dots,T_n)$-invariant decomposition \eqref{invaridecdef}. Or give restrictions either on the transformations or on $X$ (but not on the function class $\RR^X$) such that \eqref{eqdifference} becomes also sufficient.
\end{problem}

More precisely, we focus on complementary conditions, functional equations, on the functions, which they must satisfy in case of existence of an invariant decomposition \eqref{invaridecdef} and which equations will also imply
existence of such a decomposition. Difference equations and/or inequalities occur here naturally, as is also suggested by the appearance of the Whitney condition in Theorem \ref{th:LRC}.

Further necessary conditions can be easily obtained. Indeed, as the transformations commute, \eqref{eqdifference} implies
\begin{equation}\label{eqdifferencekj}
\Diff_{T_1^{k_1}}\dots \Diff_{T_n^{k_n}} f = 0 \qquad (\forall k_1,\dots,k_n\in\NN).
\end{equation}

\smallskip\noindent
Now the major difficulties come from the following features:
\begin{enumerate}[1.]
\item The transformations $T_j$ may not be invertible.
\item The ``mix-up'' of transformations can be completely
    irregular: $T^5 S^3x=T^7S^2 x$ for some $x\in X$ and
    nothing similar for other points $y\in X$.
\item Functions on $X$ lack any structure beyond the obvious linear one---no
    boundedness, continuity, measurability, compatibility
    with underlying structure of $X$, nothing---so not
    much theoretical mathematics but pure combinatorics
    can be invoked.
\end{enumerate}
For two transformations, i.e., $n=2$, the answer is completely known:
\begin{theorem}[B.{} Farkas, Sz.{} R\'ev\'esz \cite{invari}]\label{thm:twodecomp}
Let $X$ be a non-empty set, let $S,T:X\to X$ be commuting
transformations, and let $f\in\RR^X$. The following are equivalent:
\begin{enumerate}[(i)]
\item \label{enum:1:i} There exists a decomposition
    $f=g+h$, with $g$ and $h$ being $S$- and $T$-invariant,
    respectively.
\item \label{enum:1:ii}$\Diff_S\Diff_T f\eqegy 0$, and if
    for some $x \in X$ and $k,n,k',n'\in\NN$ the equality
\begin{equation} \label{eq:mixingdef}
T^k S^n x=T^{k'} S^{n'}x
\end{equation}
holds, then
\begin{equation*}
f(T^k x)=f(T^{k'} x).
\end{equation*}
\item $\Diff_S\Diff_T f\eqegy 0$, and if for some $x \in X$
    and $k,n,k',n'\in\NN$ \eqref{eq:mixingdef} holds, then
\begin{equation*}
f(S^n x)=f(S^{n'} x).
\end{equation*}
\end{enumerate}
\end{theorem}
Of course, the equivalence of (ii) and (iii) is due to symmetry, if one knows that any one of them is equivalent to (i). We do not give the proof (see \cite{invari}), but mention an idea that will be useful also below.
First we partition the set $X$ with respect to an
\emph{equivalence relation}:
$x,y\in X$ are equivalent if there exist $k,n,k',n'\in\NN$ such that  $T^k S^n x=T^{k'} S^{n'}y$. $X$ splits into equivalence classes $X/\hskip-5pt\sim$, from
which \emph{by the axiom of choice} we choose a representation
system.  Obviously, it is enough to define $g$ and $h$ on each of these
equivalence classes. Indeed, for $x\in X$ the elements $x$,
$Tx$ and $Sx$ are all equivalent, so the invariance of the
desired functions $g,h$ is decided already in the common
equivalence class. So the task is now reduced to defining the
functions $g$ and $h$ on a fixed, but arbitrary equivalence
class.

\smallskip\noindent For general $n\in \NN$, $n\geq 2$ the following difference equation type necessary conditions can be found:
\begingroup
\def\thetheorem{(\textasteriskcentered)}
\begin{condition}
For every $N\leq n$, disjoint $N$-term partition $B_1\cup
B_2\cup\cdots \cup B_N=\{1,2,\dots,n\}$, distinguished elements
$h_j\in B_j$ $(j=1,\dots,N)$, indices $0<k_j, l_j, l_j'\in
\NN$, $(j=1,\dots, N)$ and $z\in X$ once the conditions
\begin{equation}\label{eq:mcond}
T_{h_j}^{k_j} T_{i}^{l_i}z=T_{i}^{l_i'}z\qquad \mbox{for all $i\in
B_j\setminus \{h_j\}$, for all  $j=1,\dots, N$}
\end{equation}
are satisfied, then
\begin{equation}\label{eq:mconclusion}
\Diff_{T_{h_1}^{k_1}}\dots \Diff_{T_{h_{N}}^{k_N}}f(z)=0.
\end{equation}
\end{condition}
\endgroup\addtocounter{theorem}{-1}
\begin{theorem}[B.{} Farkas, Sz.{} R\'ev\'esz \cite{invari}]\label{th:necessity} Let $T_1,\dots, T_n$ be commuting transformations of $X$ and let $f$ be a real
function on $X$. In order to have a $(T_1,\dots,T_n)$-invariant
decomposition \eqref{invaridecdef} of $f$
Condition (\textasteriskcentered) is necessary.
\end{theorem}
If the blocks $B_j$ are all singletons the condition
\eqref{eq:mcond} is empty, so \eqref{eq:mconclusion} expresses
exactly \eqref{eqdifferencekj}. In particular, Condition
(\textasteriskcentered) contains \eqref{eqdifference}.

\smallskip\noindent For $n=3$ transformations Condition
(\textasteriskcentered) is not only necessary but also sufficient for the existence of invariant decompositions.
\begin{theorem}[B.{} Farkas, Sz.{} R\'ev\'esz \cite{invari}]\label{thm:three}
Suppose that $T_1$, $T_2$ and $T_3$ commute and that the
function $f$ satisfies Condition (\textasteriskcentered). Then
$f$ has a $(T_1,T_2,T_3)$-invariant decomposition.
\end{theorem}
Again the proof is combinatorially   involved, so let us just state one main ingredient, the "lift-up lemma" corresponding to Lemma \ref{l:liftequiv} above. It is proved itself in a
series of lemmas, which we do not detail here.
\begin{lemma}\label{lem:diffinvsolv} Let $T,S$ be commuting transformations of $X$ and let $g:X\to\RR$ be a function satisfying $\Diff_S g=0$. Then there exists a function $h: X\to \RR$ satisfying both $\Diff_S h= 0$ and $\Diff_T h =g$ if and only if for every $x\in X$ it holds
\begin{equation} \label{twotransfcondi}
\sum_{i=0}^{k-1} g(T^i x) =0 \quad\hbox{whenever}\quad T^kS^lx=S^{l'}x ~ \text{with some} ~k,l,l'\in \NN.
\end{equation}
\end{lemma}
\begin{problem}
Is Condition (\textasteriskcentered) equivalent to \eqref{invaridecdef} for all $n\in\NN$ ($n\geq 4$)?
\end{problem}
\subsubsection{Unrelated transformations}
If the orbits of the transformations show no recurrence a satisfactory answer can be given. The relevant notion is the following.

\begin{definition} We call two commuting transformations $S,T$ on $X$
\emph{unrelated} if $T^nS^k x= T^m S^l x$ can occur only if
$n=m$ and $k=l$. In particular, then neither of the two
transformations can have any cycles in their orbits, nor do
their joint orbits have any recurrence.
\end{definition}

If all pairs $T_i$ and $T_j$ ($1\leq i\ne j\leq n$) are unrelated, then Condition (\textasteriskcentered) degenerates, as in \eqref{eq:mcond} we necessarily have that all blocks $B_j$ are singletons. Hence Condition (\textasteriskcentered) reduces merely to \eqref{eqdifferencekj} or, equivalently, to \eqref{eqdifference}.

\begin{theorem}[B.{} Farkas, Sz.{} R\'ev\'esz \cite{invari}]\label{thm:nomix}
Suppose the transformations $T_1,\dots, T_n$ are pairwise
commuting and unrelated. Then the difference equation
\eqref{eqdifference} is equivalent to the existence of some
invariant decomposition \eqref{invaridecdef}.
\end{theorem}
\begin{proof}
Only sufficiency is to be proved. We argue by induction. The cases of small $n$ are obvious.
Let $F:=\Diff_{T_{n+1}} f$. Then $F$ satisfies a difference equation of order $n$, hence by the
inductive hypothesis we can find an invariant decomposition of
$F$ in the form $F=F_1+\cdots +F_n$, where $\Diff_{T_j} F_j \eqegy0$ for $j=1,\dots,n$.
Since the transformation are unrelated, condition \eqref{twotransfcondi} in Lemma \ref{lem:diffinvsolv}
is void, and therefore the ``lift-ups'' $f_j$ with $\Diff_{T_j}
f_j=0$, $\Diff_{T_{n+1}}f_j=F_j$ exist for all $j=1,\dots,n$.
Therefore, $f_{n+1}:=f-(f_1+\cdots+f_n)$ provides a function
satisfying $\Diff_{T_{n+1}} f_{n+1} = F -( F_1+\cdots + F_n) =0$. Thus a required decomposition of $f$ is established.
\qed\end{proof}

\subsubsection{Invertible transformations}
When the transformations $T_j$ are invertible, the situation simplifies somewhat. Denote by ${G}\subseteq X^X$ the group generated by $T_1,\dots, T_n$. As before, we work on equivalence classes, now \emph{orbits} ${O}:=\{Tx:T\in {G}\}$ for some $x\in X$,
under the action of the {transformation group ${G}$}. Given a group $G$ denote by $\langle a \rangle $ the cyclic group generated by $a$ i.e., $\langle a\rangle:= \{ a^n:n\in \ZZ\}$, and for
${H}\subseteq {G}$ let $[{H}]:=\bigcap\limits_{h\in H} \langle h
\rangle$.
\begingroup
\def\thetheorem{(\textasteriskcentered\textasteriskcentered)}
\begin{condition}\label{cond:diffmod}
For all orbits ${O}$ of
${G}$, for all partitions
\begin{equation*}B_1\cup B_2\cup\cdots
\cup
B_N=\bigl\{T_1\rest_{O},T_2\rest_{O},\dots,T_n\rest_{O}\bigr\}
\end{equation*}
and any element $S_j\in [B_j]$, $j=1,\dots, N$, we have that
\begin{equation}\label{eq:diffmod}
\Diff_{{S_1}}\dots
\Diff_{{S_N}}f\rest_{O}=0\quad\mbox{holds}.
\end{equation}
\end{condition}
\endgroup\addtocounter{theorem}{-1}

The next is the main result in this setting:
\begin{theorem}[B.{} Farkas, V.{} Harangi, T.{} Keleti, Sz.{} R\'ev\'esz \cite{FHKR}]\label{th:sufficiency}
Let $T_1,\dots,T_n$ be pairwise commuting invertible transformations
on a set $X$. Let $f:X\to \RR$ be any function. Then $f$ has a
$(T_1,T_2,\dots, T_n)$-invariant decomposition \eqref{invaridecdef} if and only if it satisfies Condition~\ref{cond:diffmod}.
\end{theorem}
The proof relies on a variant of Lemma \ref{lem:diffinvsolv}.

\subsubsection{Decompositions on groups}

Let us see some consequences. Let $G$ be a group, and let $a_1,\dots, a_n\in G$. Consider the actions of $a_1,\dots, a_n$ on $G$ as left multiplications. For a function $f:G\to \RR$ we introduce the \emph{left $a$-difference operator} $\Diff_af(x):= f(ax)-f(x)$. The function $f$ is called \emph{left $a$-invariant} (or left $a$-periodic) if $\Diff_a f=0$. Since the actions are transitive, the above result takes the following form:
\begin{corollary}
\label{cor:group} Let ${G}$ be a group and $a_1,\dots,a_n\in{G}$  pairwise commuting. Then a function
$f:{G}\to \RR$ decomposes into a sum of
left $a_j$-invariant functions, $f=f_1+\cdots+ f_n$, if and only if
for all partitions $B_1\cup B_2\cup\cdots \cup
B_N=\{a_1,\dots,a_n\}$ and for each element $b_j\in [B_j]$
\begin{equation*}
\Diff_{{b_1}}\dots \Diff_{{b_N}}f=0.
\end{equation*}
\end{corollary}
In a torsion free Abelian group $A$ for $B\subseteq A $ the generator of the cyclic group $[B]$ is uniquely determined (up to taking inverse). In \cite{invari} we called this (maybe two) element(s) the \emph{least common multiple} of the elements in $B$. For instance, with this terminology we have that the least common multiple of $1$ and $\sqrt{2}$ in the group $(\RR,+)$ is $0$. Then we have the next result:
\begin{corollary} \label{cor:abeltorsfree}
Let ${A}$ be a torsion free Abelian group and
$a_1,\dots,a_n\in {A}$. A function $f:{A}\to \RR$
decomposes into a sum of $a_j$-periodic functions, $f=f_1+\cdots+
f_n$, if and only if for all partitions $B_1\cup B_2\cup\cdots \cup
B_N=\{a_1,\dots,a_n\}$ and $b_j$ being the least common multiple of
the elements in $B_j$ one has
\begin{equation}\label{eq:b1bNDiffEq}
\Diff_{{b_1}}\dots \Diff_{{b_N}}f=0.
\end{equation}
\end{corollary}
If we specify to $A=\RR$ and take $\alpha_1,\dots,\alpha_n$ incommensurable we obtain the following result first proved in \cite{MortolaPeirone}.
\begin{corollary}[S.{} Mortola, R.{} Peirone \cite{MortolaPeirone}, B.{} Farkas, Sz.{} R\'ev\'esz \cite{invari}]
Suppose $\alpha_1,\dots,\alpha_n\in\RR$ are incommensurable. Then a function $f:\RR\to\RR$ satisfies the difference equation \eqref{eq:kern} if and only if it has periodic decomposition \eqref{eq:periodd}.
\end{corollary}
The above results remain true if one considers functions with values in torsion free  groups $\Gamma$. The proof of the following is the same as for Theorem \ref{th:sufficiency} with the new aspect that taking averages in $\Gamma$ requires some additional care.
\begin{theorem}[B.{} Farkas, V.{} Harangi, T.{} Keleti, Sz.{} R\'ev\'esz \cite{FHKR}]\label{tf2tf}
Let $A, \Gamma$ be torsion free Abelian groups and
$a_1,\dots,a_n\in A$. A function $f:A\to \Gamma$
decomposes into a sum of $a_j$-periodic functions
$f_j:A\to \Gamma$, $f=f_1+\cdots+f_n$ if and only if for
all partitions $B_1\cup B_2\cup\cdots \cup B_N=\{a_1,\dots,a_n\}$
and $b_j$ being the least common multiple of the elements in $B_j$
one has \eqref{eq:b1bNDiffEq}.
\end{theorem}
Let $A$ be a torsion free Abelian group. By applying the
previous theorem for $\Gamma = \RR$ and  for $\Gamma = \ZZ$, we obtain
that for a function $f : A \to \ZZ$ the existence of a real-valued
periodic decomposition and the existence of an integer-valued
periodic decomposition are both equivalent to the same difference equation type condition.
\begin{corollary}\label{cor:integer}
If an integer-valued function $f$ on a torsion free Abelian group
$A$ decomposes into the sum of $a_j$-periodic real-valued
functions for some $a_1,\dots,a_n$, then $f$
also decomposes into the sum of $a_j$-periodic integer-valued
functions.
\end{corollary}
There are examples showing that one cannot get rid of the torsion freeness of $A$ in Corollary \ref{cor:integer} or Theorem \ref{tf2tf}, see \cite{FHKR}.

Note that in crystallography and other applications, reconstruction or at least unique identification of integer-valued functions or characteristic functions of sets from various (partial) information concerning their Fourier transform are rather important. This also motivates the interest of integer-valued periodic decompositions or decompositions with values within a subgroup. In turn, support of a Fourier transform can reveal the existence of a periodic decomposition, see e.g. \cite[2.7 and 2.8]{KJ}, or the analogous idea of the proof for Proposition \ref{p:AP}. For more about this see \cite{KJ} and the references therein.

\section{Actions of semigroups}\label{sec:sgrpact}
Let $X$ be a non-empty set and let $T:X\to X$ be an arbitrary mapping.
If a function $f:X\to \RR$ is invariant under $T$, i.e., $\Diff_T f=0$, then it is evidently invariant under each iterate $T^n$ of $T$ for $n\in \NN$. Given commuting mappings $T_1,\dots, T_n:X\to X$ consider the generated semigroups
\begin{equation}\label{eq:cyclic}
{S}_j:=\bigl\{T_j^n:n\in \NN\bigr\}.
\end{equation}
The corresponding semigroup of the Koopman operators on $\RR^X$ is denoted by $\mathscr{S}_j$.
(Recall that we use the same symbol $T$ for the  Koopman operator of $T\in X^X$.)
For a subset $\mathscr{A}$ of linear operators on $\RR^X$ we introduce the notations $\ker \mathscr{A}:=\bigcap_{A\in \mathscr{A}}\ker A$.
Then the equality
\begin{equation}\label{eq:kerT2}
\ker(T_1-\Id)\cdots (T_n-\Id)=\ker(T_1-\Id)+\cdots+\ker (T_n-\Id)
\end{equation}
is easily seen to be equivalent to
\begin{equation}\label{eq:Ssgrpdecomp}
\ker(\mathscr{S}_1-\Id)\cdots (\mathscr{S}_n-\Id)=\ker(\mathscr{S}_1-\Id)+\cdots+\ker (\mathscr{S}_n-\Id).
\end{equation}
In what follows we study this equality when $\mathscr{S}_j$ are general, not necessarily cyclic, semigroups.

\medskip Let $S$ be a discrete semigroup with unit element, and let $S_j$, $j=1,\dots, n$ unital subsemigroups of $S$ that all act on the non-empty set $X$ (from the left), the unit acting as the identity. Suppose furthermore $st=ts$ for all $s\in S_j$ and $t\in S_i$ with $i\neq j$ (the actions of different $S_j$s are commuting).

\begin{theorem}[B.{} Farkas \cite{Fa12a}]\label{thm:Sdecomp}
Suppose that for $j=1,\dots,n$ the unital semigroups $S_j$ on the set $X$ are (right-)amenable and that the actions of the different $S_j$ are commuting. Denote by $\mathscr{S}_j$ the semigroups of the Koopman operators.
Then \eqref{eq:Ssgrpdecomp} holds in the space $\Bb(X)$.

Furthermore, if $X$ is uniform (topological) space and the action of $S_j$ on $X$ is uniformly equicontinuous, then \eqref{eq:Ssgrpdecomp} holds in the space $\BUC(X)$.
\end{theorem}
This result and its proof generalizes those of Gajda \cite{gajda:1992}, who used Banach limits (i.e., amenability of $\ZZ$ or $\NN$) to establish the above for $\ZZ$ and $\NN$ actions, i.e., for semigroups as in \eqref{eq:cyclic}. The next consequence immediately follows.
\begin{corollary}[Z.{} Gajda \cite{gajda:1992}]\label{cor:agrpdecomp}
Let $A$ be a locally compact Abelian group, and let $a_1,\dots, a_n\in A$. Then \eqref{eq:kerT2} holds in $\BUC(\RR)$ for $T_j$ being the shift operator by $a_j$. In particular $\BUC(\RR)$ has the decompositon property.
\end{corollary}
Let us finally return to the purely linear operator setting on an arbitrary Banach space $E$. A subsemigroup $\mathscr{S}\subseteq \LLL(E)$ of bounded linear operators is  called \emph{mean ergodic} if the closed convex hull $\clconv({\mathscr{S}})\subseteq \LLL(E)$ contains a zero element $P$, i.e., $PT=P=TP$ for every $T\in \mathscr{S}$. In this case $P$ is a projection, called the \emph{mean ergodic projection} of $\mathscr{S}$, and it holds  (see \cite{Nagel73})
\begin{equation*}
E=\ran P\oplus \ran(\Id-P)\quad\mbox{with}\quad \ran P=\ker(\mathscr{S}-\Id)~.
\end{equation*}
\begin{theorem}[B.{} Farkas \cite{Fa12a}]
Let $\mathscr{S}_1,\mathscr{S}_2,\dots,\mathscr{S}_n\subseteq \LLL(E)$
be mean ergodic operator {semigroups} and suppose that $ST=TS$
whenever $T\in \mathscr{S}_i$, $S\in \mathscr{S}_j$ with $i\neq j$. Then \eqref{eq:Ssgrpdecomp} holds.
\end{theorem}
Since an operator $T$ is mean ergodic if and only if the semigroup $\{T^n:n\in \NN\}$ is mean ergodic, the previous result contains Proposition \ref{p:meanerg}.  Moreover, the obvious modification of Theorem \ref{th:topvsp} (using fixed points in the closed convex hull) for this semigroup setting is easily proved, but this we will not pursue here. Furthermore, the analogue of Corollary \ref{c:vectortopology} can be formulated for amenable semigroups instead of cyclic ones, where of course one applies Day's fixed point theorem, see \cite{Day} instead of the one of Markov and Kakutani.
\begin{problem}
Does the space $\Cb(A)$ of bounded and continuous functions, where $A$ is a locally compact Abelian group,  has the decomposition property  with respect to translations? If $A$ is compact or discrete or $A=\RR$, this is so by the previous results. What about $A=\RR^2$?
\end{problem}

\section{Further results}
We briefly touch upon topics that, regrettably, 
could not be covered in detail.

First we take a second glimpse at the original problem.

\begin{theorem}[T.{} Natkaniec, W.{} Wilczy\'nski \cite{Natkaniec}]
Let $\alpha_1,\dots, \alpha_n\in \RR\setminus\{0\}$ be incommensurable. A function $f:\RR\to\RR$ has a decomposition \eqref{eq:periodd} with $f_1,\dots, f_n$ Darboux functions if and only if \eqref{eq:kern} holds.
\end{theorem}

See \cite{Natkaniec} for the proof where also the decompositon property of Marczewski measurable functions is studied for incommensurable periods. It is also shown that the identity is not the sum of periodic functions having the Baire property.
For classes of measurable real functions we have, e.g., the following.
\begin{theorem}[T.{} Keleti \cite{KeletiMeasInt}] \label{t:nomeasdec}None of the following classes $\FC$ have the decomposition property:
\begin{enumerate}[a)]
\item $\FC=\{f:\mbox{ $f:\RR\to \ZZ$, $f\in\Ell^\infty(\RR)$}\}$,
\item $\FC=\{f:\mbox{ $f:\RR\to \ZZ$ is bounded measurable}\}$,
\item  $\FC=\{f:\mbox{ $f:\RR\to \RR$ is a.e.{} integer-valued and $f\in\Ell^\infty(\RR)$}\}$,
\item $\FC=\{f:\mbox{ $f:\RR\to \ZZ$ is a.e.{} integer-valued, bounded and measurable}\}$.
\end{enumerate}
\end{theorem}
For more information on measurable decompositions see also \cite{keleti:1996,keleti:1997,KeletiMeas}.
Next we turn to integer-valued decompositions on Abelian groups. We mention only three exemplary results from \cite{KKKR}:
\begin{theorem}[Gy.{} K\'arolyi, T.{} Keleti, G.{} K\'os, I.{}Z.{} Ruzsa \cite{KKKR}]
\begin{enumerate}[a)]
\item Suppose $f:\ZZ\to\ZZ$ has an $(\alpha_1,\dots, \alpha_n)$-periodic decomposition into  real-valued functions with $a_j\in \ZZ$. Then it has an $(\alpha_1,\dots, \alpha_n)$-periodic integer-valued decomposition.
\item For $\alpha_1,\dots, \alpha_n\in \ZZ$ the class of $\ZZ\to\ZZ$ functions has the decomposition property.
\item Let $A$ be a torsion-free Abelian group. Then the class of bounded $A\to \ZZ$ functions has the decomposition property if and only if $A$ is isomorphic to an additive subgroup of $\QQ$.
\end{enumerate}
\end{theorem}
For a proof and for an abundance of further information we refer to \cite{KKKR}, and remark that part c) above implies that the class of bounded and integer-valued functions  does not have the decomposition property known also from Theorem \ref{t:nomeasdec}, see  also \cite[Cor.{} 3.4]{KKKR}.

\medskip\noindent Finally, we discuss some aspects of uniqueness of decompositions. Of course, one  cannot expect uniqueness in the original setting, since appropriate constant functions can be added to the summands in \eqref{eq:periodd} not affecting the validity of \eqref{eq:kern}. If one restricts to  certain function classes then only this trivial procedure can produce different decompositions (for incommensurable periods).
\begin{theorem}[M.{} Laczkovich, Sz.{} R\'ev\'esz \cite{laczkovich/revesz:1990}]
For incommensurable periods a periodic decomposition in $\Ell^{\infty}(\RR)$ of a function $f\in\Ell^\infty(\RR)$ is unique up to additive constants.
\end{theorem}
In the original setting of the decomposition problem, i.e., in $\RR^\RR$ the situation is somewhat more complicated. E.g. consider  $n=2$, $f=f_1+f_2$ with $\Delta_{a_j}f_j=0$, $j=1,2$. Let $h$ be a not identically $0$ function that is both $\alpha_1$- and $\alpha_2$-periodic. Then $f=(f_1+h)+(f_2-h)$ is a different decomposition.

In general two decompositions $f=g_1+\cdots+g_n$ and $f=f_1+\cdots+f_n$ with $\Diff_{\alpha_j}g_j=\Diff_{\alpha_j}f_j=0$ $j=1,\dots,n$ are called \emph{essentially the same} if there are functions $h_{ij}\in\RR^\RR$ for $i,j=1,\dots,n$ with $h_{ii}=0$, $h_{ij}=-h_{ji}$, $\Delta_{\alpha_i}h_{ij}=0$, $\Delta_{\alpha_j}h_{ij}=0$ such that for all $j=1,\dots, n$ one has $ f_j-g_j=\sum_{i=1}^n h_{ij}.$ Note that for incommensurable periods $\alpha_i/\alpha_j\not\in\QQ$ we necessarily have $h_{ij}=$constant  on each equivalence class of $\RR$ (for the equivalence relation as in the paragraph after Theorem \ref{thm:twodecomp}), whence in case of continuity on the whole real line.

Essential uniqueness of decomposition depends very much on the periods:
\begin{theorem}[V.{} Harangi \cite{HarangiFund}] For $\alpha_1,\dots, \alpha_n\in \RR\setminus\{0\}$ the following assertions are equivalent:
\begin{enumerate}[(i)]
\item If any three numbers $\alpha_i, \alpha_j, \alpha_k$ from $\alpha_1,\dots, \alpha_n$ are pairwise linearly independent over $\QQ$, then they are linearly independent over $\QQ$.
\item Any two $(\alpha_1,\dots,\alpha_n)$-periodic decomposition of  a function $f$ are essentially the same, i.e., the decomposition is essentially unique.
\item If a function $f:\RR\to \ZZ$ has a $(\alpha_1,\dots,\alpha_n)$-periodic decomposition into real-valued  functions, then it has also an  integer-valued one.
\end{enumerate}
\end{theorem}
See also \cite{HarangiRAEX}, and \cite{HarangiFund} or \cite{HarangiDipl} for details and further directions.

We end this survey by posing the following problem:
\begin{problem}
Study the periodic decomposition problem for functions $f$ on $\RR$, or on an  Abelian group, with values in $\RR$ mod $1$ (or in an Abelian group).
\end{problem}
\def\cprime{$'$}

%

\end{document}